\documentclass[]{interact}

\usepackage{epstopdf}
\usepackage[caption=false]{subfig}

\usepackage[numbers,sort&compress]{natbib}
\bibpunct[, ]{[}{]}{,}{n}{,}{,}
\renewcommand\bibfont{\fontsize{10}{12}\selectfont}
\makeatletter
\def\NAT@def@citea{\def\@citea{\NAT@separator}}
\makeatother

\theoremstyle{plain}
\newtheorem{theorem}{Theorem}[section]
\newtheorem{lemma}[theorem]{Lemma}

\theoremstyle{definition}

\theoremstyle{remark}

\begin{document}

\articletype{Article Dedicated to the memory of Professor S.~N.~Mergelyan on the occasion of his 95th birthday anniversary}

\title{Solvability of a System of Nonlinear Integral Equations on the Entire Line}

\author{
\name{A.~Kh. Khachatryan\textsuperscript{a}, Kh.~A. Khachatryan\textsuperscript{b}\thanks{CONTACT Kh.~A. Khachatryan. Email: khachatur.khachatryan@ysu.am} and H.S. Petrosyan\textsuperscript{c}}
\affil{\textsuperscript{a}Armenian National Agrarian University, Yerevan, Armenia; \textsuperscript{b}Yerevan State University, Yerevan, Armenia; \textsuperscript{c}Armenian National Agrarian University, Yerevan, Armenia}
}

\maketitle

\begin{abstract}
A system of singular integral equations with monotone and concave nonlinearity in the subcritical case is investigated. The specified system and its scalar analog have direct applications in various areas of physics and biology. In particular, scalar and vector equations of this nature are encountered in the dynamic theory of $p$-adic strings, in the theory of radiative transfer, in the kinetic theory of gases, and in the mathematical theory of the spread of epidemic diseases. A constructive theorem of existence in the space of continuous and bounded vector functions is proved. The integral asymptotic of the constructed solution is studied. An effective iterative process for constructing an approximate solution to this system is proposed. In a certain class of bounded vector functions, the uniqueness of the solution is proved. In the class of bounded vector functions (with non-negative coordinates) with a zero limit at infinity, the absence of an identically zero solution is proved. Examples of the matrix kernel and nonlinearities are given to illustrate the importance of the obtained results. Some of the examples have applications in the above-mentioned branches of natural science.
\end{abstract}


\section{Introduction}
\subsection{Statement of the problem.}
Let us consider the following class of nonlinear vector singular integral equations on $\mathbb{R}:=(-\infty,+\infty):$
\begin{equation}\label{Khachatryan1}
f_i(x)=\sum\limits_{j=1}^N\int\limits_{-\infty}^\infty K_{ij}(x-t) \mu_j (t)G_j(f_j(t))dt,\quad x\in\mathbb{R},\quad i=\overline{1,N}
\end{equation}
with respect to the desired vector function $f(x)=(f_1(x),\ldots,f_N(x))^T,$ where $T$ is the transpose sign. The solution of system \eqref{Khachatryan1} is sought in the following class of measurable vector functions $$\Omega:=\{f(x)=(f_1(x),\ldots,f_N(x))^T: f_j\in L_\infty(\mathbb{R}),\, f_j(x)\geq0, \, x\in\mathbb{R},\, j=\overline{1,N}\}.$$ The elements of the matrix kernel  $K(\tau)=(K_{ij}(\tau))_{i,j=1}^{N\times N}$ of system \eqref{Khachatryan1} satisfy the following conditions:
\begin{enumerate}
  \item [1)] $K_{ij}(-\tau)=K_{ij}(\tau)=K_{ji}(\tau)>0,$ $\tau\in\mathbb{R},$ $K_{ij}\in L_1(\mathbb{R})\cap L_\infty(\mathbb{R}),$ $i,j=\overline{1,N},$
  \item [2)] the spectral radius of the matrix  $A=\left(\int\limits_{-\infty}^\infty K_{ij}(\tau)d\tau\right)_{i,j=1}^{N\times N}$ is equal to one, and $\int\limits_0^\infty \tau K_{ij}(\tau)d\tau<+\infty,$ $i,j=\overline{1,N}.$
\end{enumerate}
Functions $\mu_j(t),$ $j=\overline{1,N}$  have the following properties:
\begin{enumerate}
  \item [a)] $\mu_j(t)>1,$ $t\in\mathbb{R},$ $j=\overline{1,N}$
  \item [b)] $\mu_j-1\in L_1^0(\mathbb{R}),$ $j=\overline{1,N},$
\end{enumerate}
where $L_1^0(\mathbb{R})$ is the space of Lebesgue-summable functions on  $\mathbb{R}$ that have zero limit at $\pm\infty.$

Denote by  $a_{ij}:=\int\limits_{-\infty}^\infty K_{ij}(\tau)d\tau,$ $i,j=\overline{1,N}.$  Then, according to Perron's theorem  (see \cite{lan1}) there exists a vector  $\eta=(\eta_1,\ldots,\eta_N)^T$ with positive coordinates  $\eta_i,$ $i=\overline{1,N}$ such that
\begin{equation}\label{Khachatryan2}
\sum\limits_{j=1}^N a_{ij}\eta_j=\eta_i, \quad i=\overline{1,N}.
\end{equation}
We fix this vector and impose the following conditions on the nonlinearities $G_j(u),$ $j=\overline{1,N}:$
\begin{enumerate}
  \item [I)] $G_j\in C[0,+\infty),$ $G_j(u)$ monotonically increases in  $u$ on the set $[0,+\infty),$ $j=\overline{1,N},$
  \item [II)] $G_j(0)=0,$ $G_j(\eta_j)=\eta_j,$ $j=\overline{1,N},$
  \item [III)] the functions  $G_j(u)$ are strictly concave on the set  $[0,+\infty),$ $j=\overline{1,N},$
  \item [IV)] there exists a mapping  $\varphi: [0,1]\rightarrow[0,1]$ with the properties:   $\varphi(0)=0,$ $\varphi(1)=1,$ $\varphi\in C[0,1],$ $\varphi$ monotonically increases on $[0,1],$ $\varphi$ concave on the segment  $[0,1]$ such that the following inequalities hold:
      $$G_j(\sigma u)\geq \varphi(\sigma)G_j(u),\quad \sigma\in[0,1], \quad u\in[\eta_j,\xi_j],\quad j=\overline{1,N},$$
      where the numbers  $\xi_j (\xi_j>\eta_j),$ $j=\overline{1,N}$ are uniquely determined from the following system of nonlinear algebraic equations:
   \begin{equation}\label{Khachatryan3}
\tau_i=\sum\limits_{j=1}^N (a_{ij}+b_{ij})G_j(\tau_j),\quad i=\overline{1,N}.
\end{equation}
In system \eqref{Khachatryan3} the elements  $b_{ij}$ are defined by the following formula
\begin{equation}\label{Khachatryan4}
b_{ij}:= \int\limits_{-\infty}^\infty (\mu_j(t)-1)dt \sup\limits_{\tau\in\mathbb{R}}(K_{ij}(\tau)),\quad i,j=\overline{1,N}.
\end{equation}
It should be noted that the existence of a unique solution  $\xi=(\xi_1,\ldots,\xi_N)^T,$ $\xi_j>\eta_j,$ $j=\overline{1,N}$ of the nonlinear system \eqref{Khachatryan3} follows from Lemmas~2.1 and 2.2 of \cite{khach2}.
\end{enumerate}
The main goal of the paper is to study the issues of existence, uniqueness, absence, and asymptotic behavior at $\pm\infty$ of a nontrivial solution of the system of nonlinear singular (the singularity is reflected in the functions $\mu_j(t),$ $j=\overline{1,N}$) integral equations  \eqref{Khachatryan1} in the class  $\Omega.$
\subsection{Possible applications and history of the question.}
System \eqref{Khachatryan1} has important practical significance in various areas of natural science. In particular, system  \eqref{Khachatryan1} and the corresponding scalar analogue of this system are encountered in the dynamical theory of $p$-adic open and open-closed strings for the scalar field of tachyons, in the theory of radiation transfer in inhomogeneous media, in the kinetic theory of gases within the framework of the modified Bhatnagar-Gross-Krook model, in the mathematical theory of the spread of epidemic diseases within the framework of the usual and modified Atkinson-Reuter and Dieckmann-Kaper models (see \cite{vla3,brek4,vla5,khach6,eng7,cerc8,khach9,diek10,diek11,atk12}). It is also interesting to note that the scalar analogue of system \eqref{Khachatryan1} has an important application in cosmology  (see \cite{aref13} and \cite{aref14}).

The scalar analogue of system  \eqref{Khachatryan1}, under various restrictions on the kernel and on nonlinearity, is studied in  \cite{vla3,vla5,khach6,diek11}, \cite{khach15,khach16,khach17,petr18,khach19,khach20}. For example, in the case when $N=1,$ $K(\tau)=\dfrac{1}{\sqrt{\pi}}e^{-\tau^2},$ $\tau\in\mathbb{R},$ $\mu(t)\equiv1,$ $t\in\mathbb{R},$ and $G(u)=\sqrt[p]{u}$ ($p>2$ is an odd number), the existence of a monotonically increasing odd continuous and bounded on $\mathbb{R}$ solution of the scalar analogue of system \eqref{Khachatryan1} was proved in \cite{vla3}.  Later, in  \cite{khach15,khach16,khach17,petr18} not  only were these results generalized for general even summable bounded kernels $K$ depending on the difference of their arguments and general concave monotone continuous nonlinearities  $G,$ with $\mu(t)\equiv1,$ $t\in\mathbb{R},$ but also uniqueness theorems for the constructed solution were proved in various subclasses of measurable and bounded functions on $\mathbb{R}$. In the case where $\mu(t)>1,$ $t\in\mathbb{R},$ $\mu-1\in L_1^0(\mathbb{R}),$ $G(u)=\sqrt[p]{u},$ and the kernel  $K$ is represented as a Gaussian distribution, the issues of existence and study of the asymptotic behavior (at $\pm\infty$) of a bounded nontrivial solution of the scalar analogue of system \eqref{Khachatryan1} were investigated in \cite{vla5}.

These results were generalized for general kernels $K$ (with properties  $1), 2),$ for $N=1$) and nonlinearities  $G$ satisfying conditions of type  $I)-III)$ (for $N=1$) in \cite{khach19} and \cite{khach20}. It should also be noted that when  $\mu_j(t)\equiv1,$ $j=\overline{1,N}$ system \eqref{Khachatryan1} was studied in  \cite{khach21} and \cite{david22} under conditions  $1), 2), I)-III).$ In the mentioned works, theorems of existence and uniqueness of an alternating monotonically non-decreasing continuous and bounded solution are proved. Moreover, in  \cite{david22} conditions are given under which system  \eqref{Khachatryan1} with $\mu_j(t)\equiv1,$ $j=\overline{1,N}$ has only trivial solutions in the class  $\Omega.$

The main tool for studying the special case of system \eqref{Khachatryan1} (i.e. when  $\mu_j(t)\equiv1,$ $j=\overline{1,N}$) was the well-known theorem of M.A. Krasnosel'skii on a fixed point of a nonlinear monotone continuous operator leaving a certain conical segment invariant in regular cones (see \cite{kras23}).
\subsection{Summary of the obtained results.}
In this paper, we will study the issues of constructive existence, uniqueness and absence of a non-trivial non-negative continuous and bounded solution for the singular system \eqref{Khachatryan1} with boundary conditions on $\pm\infty:$
\begin{equation}\label{Khachatryan5}
\lim\limits_{x\rightarrow\pm\infty}f_i(x)=\alpha_i^\pm\geq0,\quad i=\overline{1,N}.
\end{equation}
In particular, it will be proved that if  $\alpha_i^\pm=0,$ $i=\overline{1,N},$ then problem  \eqref{Khachatryan1}, \eqref{Khachatryan5} in the class  $\Omega$ has
only a zero solution. In the case when there exists  an index
 $i\in\{1,2,\ldots,N\}$ such that $\alpha_i^\pm>0,$ we will prove the existence of a unique non-negative, non-trivial bounded and continuous solution to problem \eqref{Khachatryan1}, \eqref{Khachatryan5}. Moreover, it will be established that  $\alpha_i^\pm=\eta_i$ and $f_i-\eta_i\in L_1(\mathbb{R}),$ $i=\overline{1,N}.$ At the end of the paper, we will give specific particular examples of matrix kernels $K(\tau),$ functions $\mu_j(t),$ $j=\overline{1,N}$ and nonlinearities $G_j(u),$ $j=\overline{1,N}$ satisfying all the restrictions of the proved statements. Some of the given examples will also have an applied character. Note that the existence theorem will be of a constructive nature, since we propose a special iterative process that uniformly converges to a continuous solution of system \eqref{Khachatryan1} with the speed of a geometric progression.
\section{Notations and auxiliary facts}
\subsection{A priori estimates.}
The following a priori estimate for an arbitrary solution  $f\in\Omega$ of system  \eqref{Khachatryan1} plays an important role in further reasoning. The following Lemma holds:
\begin{lemma}
Under conditions  $1), 2), a), b)$ and $I)-III)$ the coordinates of any solution  $f\in\Omega$ of system  \eqref{Khachatryan1} satisfy the following  estimation from above:
\begin{equation}\label{Khachatryan6}
f_i(x)\leq\xi_i,\quad x\in\mathbb{R},\quad i=\overline{1,N}.
\end{equation}
\end{lemma}
\begin{proof}
Let us denote by  $c_i:=\sup\limits_{x\in\mathbb{R}}f_i(x),$ $i=\overline{1,N}.$ Since $f\in\Omega,$ then from  \eqref{Khachatryan1} by conditions  $I), II), 1), 2), a)$ and $b)$ we have
$$0\leq f_i(x)\leq \sum\limits_{j=1}^N \int\limits_{-\infty}^\infty K_{ij}(x-t)\mu_j(t)G_j(c_j)dt\leq$$
$$\leq \sum\limits_{j=1}^N G_j(c_j) \left( \int\limits_{-\infty}^\infty K_{ij}(x-t)dt+ \int\limits_{-\infty}^\infty (\mu_j(t)-1)dt \sup\limits_{\tau\in\mathbb{R}}(K_{ij}(\tau))\right)=$$
$$=\sum\limits_{j=1}^N(a_{ij}+b_{ij})G_j(c_j),\quad x\in\mathbb{R},\quad i=\overline{1,N},$$
whence it follows that
\begin{equation}\label{Khachatryan7}
c_i\leq\sum\limits_{j=1}^N(a_{ij}+b_{ij})G_j(c_j),\quad  i=\overline{1,N}.
\end{equation}
We now prove that  $c_i\leq \xi_i,$ $i=\overline{1,N},$ where $\xi=(\xi_1,\ldots,\xi_N)^T,$ $(\xi_j>\eta_j,\ j=\overline{1,N})$ is the unique solution of system  \eqref{Khachatryan3}.

Indeed, firstly, from  \eqref{Khachatryan3} and \eqref{Khachatryan7} it follows that
$$0\leq c_i\leq \sum\limits_{j=1}^N(a_{ij}+b_{ij})G_j(\xi_j)\max\limits_{1\leq j\leq N}\left(\frac{G_j(c_j)}{G_j(\xi_j)}\right)=\xi_i \max\limits_{1\leq j\leq N}\left(\frac{G_j(c_j)}{G_j(\xi_j)}\right),\ i=\overline{1,N}.$$
whence we obtain
\begin{equation}\label{Khachatryan8}
0\leq \frac{c_i}{\xi_i}\leq \max\limits_{1\leq j\leq N} \left(\frac{G_j(c_j)}{G_j(\xi_j)}\right), \ \ i=\overline{1,N}.
\end{equation}
Obviously, there exists an index  $j_0\in \{1,2,\ldots,N\}$ such that
\begin{equation}\label{Khachatryan9}
\max\limits_{1\leq j\leq N} \left(\frac{G_j(c_j)}{G_j(\xi_j)}\right)= \frac{G_{j_0}(c_{j_0})}{G_{j_0}(\xi_{j_0})}.
\end{equation}
In inequality \eqref{Khachatryan8} taking $i=j_0$ and considering  \eqref{Khachatryan9} we obtain
\begin{equation}\label{Khachatryan10}
0\leq \frac{c_{j_0}}{\xi_{j_0}}\leq \frac{G_{j_0}(c_{j_0})}{G_{j_0}(\xi_{j_0})}.
\end{equation}
Note that from conditions  $I)-III)$ it immediately follows that the functions $y=\frac{G_j(u)}{u},$ $j=\overline{1,N}$ monotonically decrease on the  $(0,+\infty).$ Therefore, from  \eqref{Khachatryan10} we arrive at the inequality
\begin{equation}\label{Khachatryan11}
0\leq c_{j_0}\leq \xi_{j_0}.
\end{equation}
Since the functions  $G_j(u),$ $j=\overline{1,N}$ are monotonically increasing on  $[0,+\infty),$ from \eqref{Khachatryan8}, \eqref{Khachatryan9} and \eqref{Khachatryan11} we get
$$0\leq \frac{c_{i}}{\xi_{i}}\leq \frac{G_{j_0}(c_{j_0})}{G_{j_0}(\xi_{j_0})}\leq1,\quad i=\overline{1,N}.$$
Thus $f_i(x)\leq c_i\leq\xi_i,$ $x\in\mathbb{R},$ $i=\overline{1,N}$ and the lemma is proved.
\end{proof}
The following Lemma also holds
\begin{lemma}
Let conditions  $1), 2), a), b), I)$ and $II)$ are fulfilled.  Then, if there exists $j_0\in\{1,2,\ldots, N\}$ such that  $\alpha_{j_0}^\pm>0,$ then for any solution  $f\in\Omega$ of problem  \eqref{Khachatryan1}, \eqref{Khachatryan5} the following inequalities from below hold:
\begin{equation}\label{Khachatryan12}
\beta_j:=\inf\limits_{x\in\mathbb{R}}f_j(x)>0,\quad j=\overline{1,N}.
\end{equation}
\end{lemma}
\begin{proof}
First, note that from condition $\alpha_{j_0}^\pm>0$ it immediately follows that there exists a number $r>0$ such that for $|x|>r$ the following inequality holds
\begin{equation}\label{Khachatryan13}
f_{j_0}(x)>\max \left(\frac{\alpha_{j_0^+}}{2}, \frac{\alpha_{j_0^-}}{2}\right)=:\delta>0.
\end{equation}
Taking into account  \eqref{Khachatryan13}, as well as conditions  $1), a), I)$ and $II)$ from \eqref{Khachatryan1} we obtain
\begin{equation}\label{Khachatryan14}
\begin{array}{c}
\displaystyle f_i(x)\geq \int\limits_{-\infty}^\infty K_{ij_0}(x-t)\mu_{j_0}(t)G_{j_0}(f_{j_0}(t))dt\geq \\
\displaystyle\geq G_{j_0}(\delta) \int\limits_{|t|>r}K_{ij_0}(x-t)dt>0,\quad x\in\mathbb{R},\quad i=\overline{1,N}.
\end{array}
\end{equation}
We now verify that  $f_i\in C(\mathbb{R}),$ $i=\overline{1,N}.$ To prove these inclusions, we use the following known fact in the operation convolutions:  if  $\psi_1\in L_1(\mathbb{R}),$ $\psi_2\in L_\infty(\mathbb{R}),$ that $(\psi_1*\psi_2)(x):= \int\limits_{-\infty}^\infty \psi_1(x-t)\psi_2(t)dt$ represents continuous function  on set   $\mathbb{R}$ (see \cite{rud24}). Using the last fact, conditions  $1), b), I), II)$ and the fact that  $f\in\Omega,$ from the relations
$$f_i(x)=\sum\limits_{j=1}^N \int\limits_{-\infty}^\infty K_{ij}(t)G_{j}(f_{j}(x-t))(\mu_{j}(x-t)-1)dt+$$
$$+\sum\limits_{j=1}^N \int\limits_{-\infty}^\infty K_{ij}(x-t)G_{j}(f_{j}(t))dt,\quad x\in\mathbb{R},\quad i=\overline{1,N}$$
we arrive at the inclusions
\begin{equation}\label{Khachatryan15}
f_i\in C(\mathbb{R}),\quad i=\overline{1,N}.
\end{equation}
From \eqref{Khachatryan14} and \eqref{Khachatryan15} according to the Weierstrass theorem, we obtain that
\begin{equation}\label{Khachatryan16}
\min\limits_{x\in[-r,r]}f_{j_0}(x)>0.
\end{equation}
Thus, from  \eqref{Khachatryan13} and \eqref{Khachatryan16} it follows that
\begin{equation}\label{Khachatryan17}
\beta_{j_0}=\inf\limits_{x\in\mathbb{R}}f_{j_0}(x)>0.
\end{equation}
Using inequality  \eqref{Khachatryan17} and conditions  $1), 2), a), I), II),$ from \eqref{Khachatryan1} we obtain
$$f_i(x)\geq G_{j_0}(\beta_{j_0}) \int\limits_{-\infty}^\infty K_{ij_0}(x-t)\mu_{j_0}(t)dt>  G_{j_0}(\beta_{j_0}) \int\limits_{-\infty}^\infty K_{ij_0}(x-t)dt=$$$$= G_{j_0}(\beta_{j_0}),\ x\in\mathbb{R},\ i=\overline{1,N}, $$
whence it follows that  $\beta_i\geq G_{j_0}(\beta_{j_0})>0,$ $i=\overline{1,N}.$ Thus, the lemma is completely proved.
\end{proof}
Using Lemma~2.2, we can prove the following more precise estimation from below for the quantities   $\beta_i,$ $i=\overline{1,N}.$
\begin{lemma}
Under the conditions of Lemma~2.2, if condition  $III),$  is also satisfied, then the following estimation from below  holds
\begin{equation}\label{Khachatryan18}
\beta_i\geq \eta_i,\quad i=\overline{1,N}.
\end{equation}
\end{lemma}
\begin{proof}
Using  \eqref{Khachatryan12}, as well as conditions  $1), 2), a), I)$ and $II)$ from \eqref{Khachatryan1} we have
\begin{equation}\label{Khachatryan19}
f_i(x)\geq \sum\limits_{j=1}^N a_{ij}G_j(\beta_j),\quad x\in\mathbb{R},\quad i=\overline{1,N},
\end{equation}
from which it follows that
\begin{equation}\label{Khachatryan20}
\beta_i\geq \sum\limits_{j=1}^N a_{ij}G_j(\beta_j),\quad i=\overline{1,N}.
\end{equation}
Taking into account  \eqref{Khachatryan2} from \eqref{Khachatryan20} we obtain
\begin{equation}\label{Khachatryan21}
\beta_i\geq \eta_i \min\limits_{1\leq j\leq N}\left(\frac{G_j(\beta_j)}{\eta_j}\right), \quad i=\overline{1,N}.
\end{equation}
Note  that there exists an index  $j^*\in \{1,2,\ldots, N\}$ such that
\begin{equation}\label{Khachatryan22}
\min\limits_{1\leq j\leq N}\left(\frac{G_j(\beta_j)}{\eta_j}\right)=\frac{G_{j^*}(\beta_{j^*})}{\eta_{j^*}}.
\end{equation}
In inequality  \eqref{Khachatryan21} taking  $i=j^*$ and considering  \eqref{Khachatryan22} we obtain that $\beta_{j^*}\geq G_{j^*}(\beta_{j^*}).$  Since  $\dfrac{G_{j^*}(u)}{u}$ monotonically decreases on  $(0,+\infty),$ then from the obtained inequality and the relation $G_{j^*}(\eta_{j^*})=\eta_{j^*}$ it follows that  $\beta_{j^*}\geq \eta_{j^*}.$  Since  $G_{j^*}(u)$ monotonically increases on the set  $[0,+\infty),$ then from the last inequality we obtain that  $G_{j^*}(\beta_{j^*})\geq G_{j^*}(\eta_{j^*})=\eta_{j^*}.$ Therefore, taking into account \eqref{Khachatryan22} from \eqref{Khachatryan21} we arrive at  \eqref{Khachatryan18}. The lemma is proved.
\end{proof}
\subsection{On some inclusions for arbitrary solutions of problem  \eqref{Khachatryan1}, \eqref{Khachatryan5}.}
The following Lemma holds
\begin{lemma}
Let conditions  $1), 2), a)$ and $I)-III)$ be satisfied. If $\alpha_i^+=0,$ $i=\overline{1,N},$ then an arbitrary solution  $f\in\Omega$ of problem  \eqref{Khachatryan1}, \eqref{Khachatryan5} is a summable function on $[0,+\infty)$.
\end{lemma}
\begin{proof}
Since $\alpha_i^+=0,$ $i=\overline{1,N},$  then there exists a number  $r>0$ such that for $x>r$ the inequality holds
\begin{equation}\label{Khachatryan23}
f_i(x)<\frac{\eta_i}{2},\quad i=\overline{1,N}.
\end{equation}
First, we prove that  $f_i\in L_1(r,+\infty),$ $i=\overline{1,N}.$ Taking into account  \eqref{Khachatryan2}, conditions $1), 2), a),$ from \eqref{Khachatryan1} due to  \eqref{Khachatryan23} for all $x\in(r,+\infty)$ we have
$$0<\eta_i-f_i(x)=\sum\limits_{j=1}^N a_{ij}\eta_j-\sum\limits_{j=1}^N \int\limits_{-\infty}^\infty K_{ij}(x-t)\mu_j(t)G_j(f_j(t))dt=$$
$$= \sum\limits_{j=1}^N \int\limits_{-\infty}^\infty K_{ij}(x-t)(\eta_j-G_j(f_j(t)))dt- \sum\limits_{j=1}^N \int\limits_{-\infty}^\infty K_{ij}(x-t)(\mu_j(t)-1)G_j(f_j(t))dt\leq$$
$$\leq \sum\limits_{j=1}^N \int\limits_{-\infty}^\infty K_{ij}(x-t)(\eta_j-G_j(f_j(t)))dt,\quad x\in(r,+\infty),\quad i=\overline{1,N}.$$
 Let  $r_0\geq r$ be an arbitrary number. Integrating both parts of the inequality
\begin{equation}\label{Khachatryan24}
0<\eta_i-f_i(x)\leq \sum\limits_{j=1}^N \int\limits_{-\infty}^\infty K_{ij}(x-t)(\eta_j-G_j(f_j(t)))dt,\quad x>r,\quad  i=\overline{1,N}
\end{equation}
 over  $x$ in the interval from  $r$ to $r_0$ and using Lemma~2.1, conditions  $1), 2), I)$ and $II)$ we get
$$0<\int\limits_r^{r_0}(\eta_i-f_i(x))dx\leq \sum\limits_{j=1}^N \int\limits_r^{r_0} \int\limits_{-\infty}^\infty K_{ij}(x-t)(\eta_j-G_j(f_j(t)))dt dx=$$
$$=\sum\limits_{j=1}^N \int\limits_r^{r_0} \int\limits_{-\infty}^0 K_{ij}(x-t)(\eta_j-G_j(f_j(t)))dt dx+\sum\limits_{j=1}^N \int\limits_r^{r_0} \int\limits_{0}^{r_0} K_{ij}(x-t)(\eta_j-G_j(f_j(t)))dt dx+$$
$$+ \sum\limits_{j=1}^N \int\limits_r^{r_0} \int\limits_{r_0}^\infty K_{ij}(x-t)(\eta_j-G_j(f_j(t)))dt dx\leq \sum\limits_{j=1}^N \eta_j \int\limits_0^{r_0} \int\limits_{-\infty}^0 K_{ij}(x-t)dtdx+$$
$$+ \sum\limits_{j=1}^N \eta_j \int\limits_0^{r_0} \int\limits_{r_0}^\infty K_{ij}(x-t)dtdx+ \sum\limits_{j=1}^N \eta_j \int\limits_r^{r_0} \int\limits_0^r K_{ij}(x-t)dtdx+$$
$$+\sum\limits_{j=1}^N \int\limits_r^{r_0} \int\limits_{r}^{r_0} K_{ij}(x-t)(\eta_j-G_j(f_j(t)))dt dx= \sum\limits_{j=1}^N \eta_j \int\limits_0^{r_0} \int\limits_{x}^\infty K_{ij}(\tau)d\tau dx+$$
$$+ \sum\limits_{j=1}^N \eta_j \int\limits_0^{r_0} \int\limits_{r_0-x}^\infty K_{ij}(\tau)d\tau dx+
+ \sum\limits_{j=1}^N \eta_j \int\limits_r^{r_0} \int\limits_{x-r}^x K_{ij}(\tau)d\tau dx+$$
$$+\sum\limits_{j=1}^N \int\limits_r^{r_0} \int\limits_{r}^{r_0} K_{ij}(x-t)(\eta_j-G_j(f_j(t)))dt dx:=I_1^i+I_2^i+I_3^i+I_4^i,\quad i=\overline{1,N}.$$
Let us separately estimate the terms  $I_1^i, I_2^i$ and $I_3^i.$  Taking into account conditions  $1), 2)$ and considering Fubini’s theorem  (see \cite{kol25}) we have
$$ I_1^i:=\sum\limits_{j=1}^N \eta_j \int\limits_0^{r_0} \int\limits_{x}^\infty K_{ij}(\tau)d\tau dx= \sum\limits_{j=1}^N \eta_j \int\limits_0^{r_0} \int\limits_{x}^{r_0} K_{ij}(\tau)d\tau dx+$$
$$+ \sum\limits_{j=1}^N \eta_j \int\limits_0^{r_0} \int\limits_{r_0}^\infty K_{ij}(\tau)d\tau dx= \sum\limits_{j=1}^N \eta_j \int\limits_0^{r_0} K_{ij}(\tau)\tau d\tau+\sum\limits_{j=1}^N \eta_j \int\limits_{r_0}^\infty K_{ij}(\tau)r_0 d\tau\leq $$
$$\leq \sum\limits_{j=1}^N \eta_j \int\limits_0^{r_0} K_{ij}(\tau)\tau d\tau+ \sum\limits_{j=1}^N \eta_j \int\limits_{r_0}^\infty K_{ij}(\tau)\tau d\tau= \sum\limits_{j=1}^N \eta_j \int\limits_0^{\infty} K_{ij}(\tau)\tau d\tau<+\infty,$$
$$I_2^i:= \sum\limits_{j=1}^N \eta_j \int\limits_0^{r_0} \int\limits_{r_0-x}^\infty K_{ij}(\tau)d\tau dx= \sum\limits_{j=1}^N \eta_j \int\limits_0^{r_0} \int\limits_{t}^\infty K_{ij}(\tau)d\tau dt\leq \sum\limits_{j=1}^N \eta_j \int\limits_0^{\infty}  K_{ij}(\tau)\tau d\tau<+\infty,$$
$$I_3^i:= \sum\limits_{j=1}^N \eta_j \int\limits_r^{r_0} \int\limits_{x-r}^x K_{ij}(\tau)d\tau dx\leq \sum\limits_{j=1}^N \eta_j \int\limits_r^{r_0} \int\limits_{x-r}^\infty K_{ij}(\tau)d\tau dx=$$
$$=\sum\limits_{j=1}^N \eta_j \int\limits_0^{r_0-r} \int\limits_{t}^\infty K_{ij}(\tau)d\tau dxdt\leq \sum\limits_{j=1}^N \eta_j \int\limits_{0}^\infty K_{ij}(\tau)\tau d\tau<+\infty.$$
Therefore, considering the above obtained inequality
$$0<\int\limits_r^{r_0} (\eta_i-f_i(x))dx\leq I_1^i+I_2^i+I_3^i+I_4^i,\quad i=\overline{1,N},$$
we arrive at the estimates
\begin{equation}\label{Khachatryan25}
\begin{array}{c}
\displaystyle 0< \int\limits_r^{r_0} (\eta_i-f_i(x))dx\leq 3 \sum\limits_{j=1}^N \eta_j \int\limits_0^{\infty}K_{ij}(\tau)\tau d\tau+ \\
\displaystyle+ \sum\limits_{j=1}^N \int\limits_r^{r_0} \int\limits_{r}^{r_0} K_{ij}(x-t)(\eta_j-G_j(f_j(t)))dt dx, \quad i=\overline{1,N}.
\end{array}
\end{equation}
We multiply both parts of  \eqref{Khachatryan25} by $\eta_i$ and then sum the obtained inequalities over all $i=1,\ldots,N.$ As a result, using relation \eqref{Khachatryan2} and the symmetry of the matrix  $A$ we arrive at the inequality
$$\sum\limits_{i=1}^N\eta_i\int\limits_r^{r_0} (\eta_i-f_i(x))dx\leq 3 \sum\limits_{i=1}^N\eta_i\sum\limits_{j=1}^N \eta_j \int\limits_0^{\infty}K_{ij}(\tau)\tau d\tau+\sum\limits_{j=1}^N  \eta_j\int\limits_r^{r_0} (\eta_j-G_j(f_j(x)))dx,$$
from which we obtain that
\begin{equation}\label{Khachatryan26}
0\leq \sum\limits_{i=1}^N \eta_i\int\limits_r^{r_0} (G_i(f_i(x))-f_i(x))dx\leq 3 \sum\limits_{i=1}^N\eta_i\sum\limits_{j=1}^N \eta_j \int\limits_0^{\infty}K_{ij}(\tau)\tau d\tau.
\end{equation}
Since \eqref{Khachatryan23} holds for all $x\in(r,+\infty),$ then due to the monotonicity of  $\dfrac{G_i(u)}{u}$ on $(0,+\infty),$ $i=\overline{1,N}$  we can assert that
\begin{equation}\label{Khachatryan27}
G_i(f_i(x))\geq \frac{G_i(\frac{\eta_i}{2})}{\frac{\eta_i}{2}}f_i(x),\quad x\in(r,+\infty),\quad i=\overline{1,N}.
\end{equation}
From \eqref{Khachatryan26} and \eqref{Khachatryan27} we obtain that
\begin{equation}\label{Khachatryan28}
2\sum\limits_{i=1}^N \left(G_i\left(\frac{\eta_i}{2}\right)-\frac{\eta_i}{2}\right)\int\limits_r^{r_0} f_i(x)dx\leq 3 \sum\limits_{i=1}^N\eta_i\sum\limits_{j=1}^N \eta_j \int\limits_0^{\infty}K_{ij}(\tau)\tau d\tau.
\end{equation}
Since $G_i\left(\dfrac{\eta_i}{2}\right)-\dfrac{\eta_i}{2}>0,$ $i=\overline{1,N}$ (see conditions  $I)-III)$), then from  \eqref{Khachatryan28}, in particular, it follows that
\begin{equation}\label{Khachatryan29}
0\leq \int\limits_r^{r_0} f_i(x)dx\leq \sum\limits_{i=1}^N\int\limits_r^{r_0} f_i(x)dx\leq \dfrac{3}{2} \frac{\sum\limits_{i=1}^N\eta_i\sum\limits_{j=1}^N \eta_j \int\limits_0^{\infty}K_{ij}(\tau)\tau d\tau}{\min\limits_{1\leq i\leq N} \left(G_i\left(\frac{\eta_i}{2}\right)-\frac{\eta_i}{2}\right)}.
\end{equation}
In \eqref{Khachatryan29} by letting  $r_0\rightarrow+\infty$ we obtain that  $f_i\in(r,+\infty),$ $i=\overline{1,N}$ and
\begin{equation}\label{Khachatryan30}
0\leq \int\limits_r^{\infty} f_i(x)dx\leq \sum\limits_{i=1}^N\int\limits_r^{\infty} f_i(x)dx\leq \dfrac{3}{2} \frac{\sum\limits_{i=1}^N\eta_i\sum\limits_{j=1}^N \eta_j \int\limits_0^{\infty}K_{ij}(\tau)\tau d\tau}{\min\limits_{1\leq i\leq N} \left(G_i\left(\frac{\eta_i}{2}\right)-\frac{\eta_i}{2}\right)}.
\end{equation}
Since $f_i\in C(\mathbb{R}),$ $i=\overline{1,N}$ (see \eqref{Khachatryan15}), then by virtue of \eqref{Khachatryan30} we conclude that  $f_i\in L_1(0,+\infty),$ $i=\overline{1,N}.$ Thus, the lemma is proved.
\end{proof}
By making similar arguments, we can also prove the validity of the following lemma:
\begin{lemma}
  Let conditions  $1), 2), a)$ and $I)-III)$ be satisfied. Then, if  $\alpha_i^-=0,$ $i=\overline{1,N},$ then an arbitrary solution  $f\in\Omega$ of problem  \eqref{Khachatryan1}, \eqref{Khachatryan5} is a summable function on the set  $(-\infty,0].$
\end{lemma}
The following Lemma  also holds:
\begin{lemma}
Let conditions  $1), 2), a), b), I)-III)$ are fulfilled and there exists an index $j_0\in\{1,2,\ldots,N\}$ such that $\alpha_{j_0}^\pm>0.$ Then $f_i-\eta_i\in L_1^0(\mathbb{R}),$ $i=\overline{1,N}$  for any solution  $f\in\Omega$ of problem  \eqref{Khachatryan1}, \eqref{Khachatryan5}.
\end{lemma}
\begin{proof}
By making similar arguments as in the proof of Lemma~2.4 and taking into account the assertion of Lemma~2.3, we can verify that there exists  $0<C=const$ such that
\begin{equation}\label{Khachatryan31}
0\leq \int\limits_{-\infty}^\infty (f_i(x)-G_i(f_i(x)))dx\leq C<+\infty,\quad i=\overline{1,N}.
\end{equation}
On the other hand, from conditions  $I)-III)$ and inequalities  \eqref{Khachatryan6} and \eqref{Khachatryan18} we have  (see Fig. 1)
$$0\leq G_i(f_i(x))-\eta_i\leq \frac{\eta_i-G_i\left(\frac{\eta_i}{2}\right)}{\frac{\eta_i}{2}}(f_i(x)-\eta_i),\quad x\in\mathbb{R},\quad i=\overline{1,N},$$
whence it follows that
\begin{equation}\label{Khachatryan32}
f_i(x)-G_i(f_i(x))\geq (f_i(x)-\eta_i)\left(\frac{G_i\left(\frac{\eta_i}{2}\right)-\frac{\eta_i}{2}}{\frac{\eta_i}{2}}\right),\quad x\in\mathbb{R},\quad i=\overline{1,N}.
\end{equation}
Since $G_i\left(\frac{\eta_i}{2}\right)>\frac{\eta_i}{2}>0,$ $i=\overline{1,N},$ then from  \eqref{Khachatryan31}, \eqref{Khachatryan32} and \eqref{Khachatryan18} we arrive at the inclusions $f_i-\eta_i\in L_1(\mathbb{R}),$ $i=\overline{1,N},$
and it is easy to prove that for
\begin{equation}\label{Khachatryan33}
\mbox{ }\quad \alpha_{j_0}^+>0,\quad f_i-\eta_i\in L_1(0,+\infty), \quad \mbox{and for }\quad \alpha_{j_0}^->0,\quad f_i-\eta_i\in L_1(-\infty,0),
\end{equation}

\begin{figure}[!h]
  \centering
\includegraphics[width=5in]{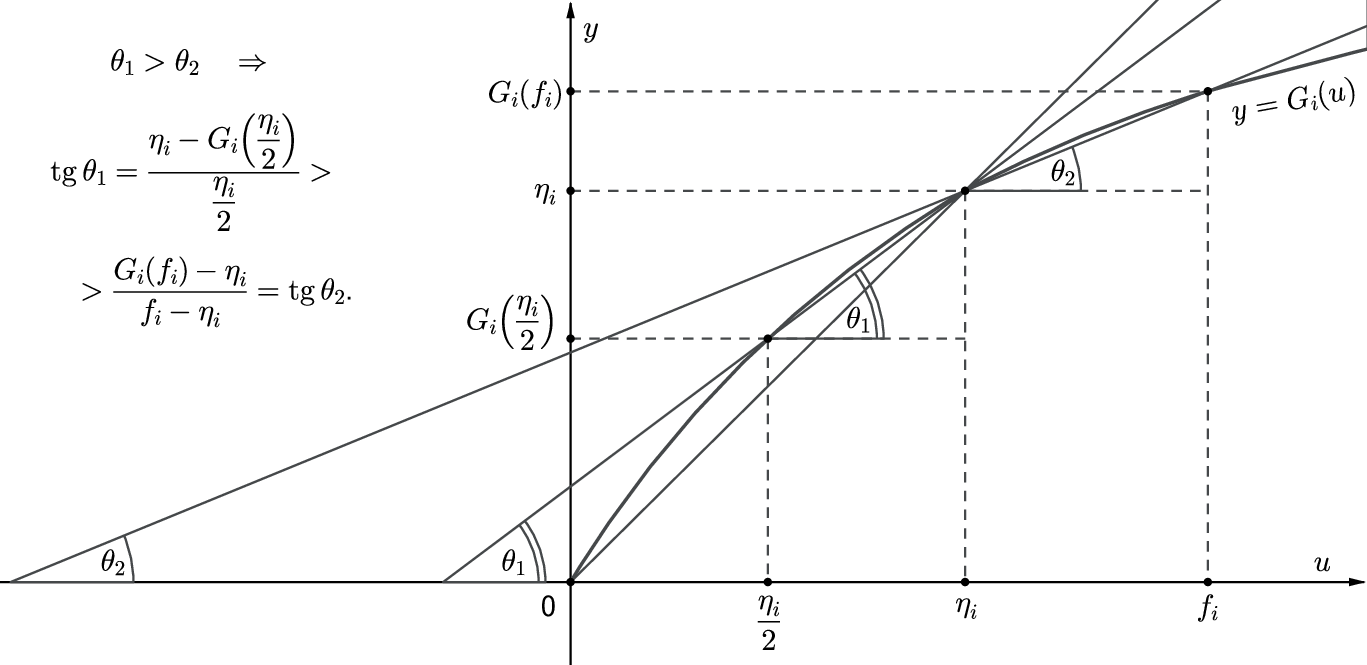}\\
  \caption{Intersection of the graph of function $y=G_i(u)$ with a line passing through points $(\eta_i, \eta_i)$ and $(\frac{\eta_i}{2},G_i(\frac{\eta_i}{2})).$}
\end{figure}

To complete the proof of the formulated lemma, it remains to check the validity of the limit relations: $\lim\limits_{x\rightarrow\pm\infty}f_i(x)=\eta_i,$ $i=\overline{1,N}.$  For this purpose, below we will use the following limit relations in the convolution operation  (see \cite{arb26} and \cite{eng27}):
\begin{enumerate}
  \item [$q_1)$] if $F_j\in L_1(\mathbb{R})\cap L_\infty (\mathbb{R}),$ $j=1,2,$ then $$(F_1*F_2)(x)=\int\limits_{-\infty}^\infty F_1(x-t)F_2(t)dt\rightarrow0, \quad \mbox{as}\quad x\rightarrow\pm\infty,$$
  \item [$q_2)$] if $F_1\in L_1(\mathbb{R})\cap L_\infty (\mathbb{R}),$ $F_2\in L_1^0(\mathbb{R}),$ then $(F_1*F_2)(x)\rightarrow0,$ as $x\rightarrow\pm\infty.$
\end{enumerate}
Using $q_1), q_2),$ \eqref{Khachatryan2}, \eqref{Khachatryan6}, as well as conditions  $1), 2), a), b), I)-III)$ from \eqref{Khachatryan1} we have
$$0\leq f_i(x)-\eta_i=\sum\limits_{j=1}^N \int\limits_{-\infty}^\infty K_{ij}(x-t)\mu_j(t)G_j(f_j(t))dt-\sum\limits_{j=1}^N\eta_j \int\limits_{-\infty}^\infty K_{ij}(x-t)dt\leq$$
$$\leq \sum\limits_{j=1}^N G_j(\xi_j) \int\limits_{-\infty}^\infty K_{ij}(x-t)(\mu_j(t)-1)dt+ \sum\limits_{j=1}^N \int\limits_{-\infty}^\infty K_{ij}(x-t)(G_j(f_j(t))-\eta_j)dt\leq$$
$$\leq \sum\limits_{j=1}^N G_j(\xi_j) \int\limits_{-\infty}^\infty K_{ij}(x-t)(\mu_j(t)-1)dt+ \sum\limits_{j=1}^N \int\limits_{-\infty}^\infty K_{ij}(x-t)(f_j(t)-\eta_j)dt\rightarrow0,$$
 when $x\rightarrow\pm\infty,\, i=\overline{1,N},$ which implies that $\lim\limits_{x\rightarrow\pm\infty}f_i(x)=\eta_i,\, i=\overline{1,N}.$ The lemma is proven.
\end{proof}
The following Lemma  is useful
\begin{lemma}
Let conditions  $1), 2), a), b), I)-III),$ $\alpha_i^-=0$ and $\alpha_i^+>0,$ $i=\overline{1,N}$ be satisfied. Then any monotonically non-decreasing solution  $f\in\Omega$ of problem  \eqref{Khachatryan1}, \eqref{Khachatryan5} satisfies the inequality:
\begin{equation}\label{Khachatryan34}
f_i(x)\leq \eta_i,\quad x\in\mathbb{R},\quad i=\overline{1,N}.
\end{equation}
Moreover, $\lim\limits_{x\rightarrow+\infty}f_i(x)=\eta_i,\, i=\overline{1,N}.$
\end{lemma}
\begin{proof}
Using the following limit relation in the convolution operation (see \cite{eng27}):
\begin{enumerate}
  \item [$q_3)$] if $F_1\in L_1(\mathbb{R})\cap L_\infty (\mathbb{R}),$ $0\leq F_2\in L_\infty(\mathbb{R})$ and there exists  $\lim\limits_{x\rightarrow\pm\infty}F_2(x)<+\infty,$ then  $$(F_1*F_2)(x)\rightarrow\int\limits_{-\infty}^\infty F_1(x)dx \cdot\lim\limits_{x\rightarrow\pm\infty}F_2(x), \quad \mbox{when}\quad x\rightarrow\pm\infty,$$  while writing system  \eqref{Khachatryan1} as follows form
\end{enumerate}
\begin{equation}\label{Khachatryan35}
\begin{array}{c}
\displaystyle f_i(x)=\sum\limits_{j=1}^N \int\limits_{-\infty}^\infty K_{ij}(x-t)(\mu_j(t)-1)G_j(f_j(t))dt+ \\
\displaystyle+\sum\limits_{j=1}^N \int\limits_{-\infty}^\infty K_{ij}(x-t)G_j(f_j(t))dt,\quad x\in\mathbb{R},\quad i=\overline{1,N}
\end{array}
\end{equation}
and taking into account conditions  $1), 2), a), b), I), II)$ from \eqref{Khachatryan35} for the limiting values $\alpha_i^+,$ $i=\overline{1,N}$ we obtain the following system of nonlinear algebraic equations
\begin{equation}\label{Khachatryan36}
\alpha_i^+=\sum\limits_{j=1}^N  a_{ij}G_j(\alpha_j^+),\quad i=\overline{1,N}.
\end{equation}
On the other hand, it follows from the results of \cite{khach28} that the unique solution of system \eqref{Khachatryan36} in the class of vectors
$$T:=\{\tau=(\tau_1,\ldots,\tau_N)^T,\, \tau_i\geq0,\,i=\overline{1,N}\,\ \mbox{and}\, \ \exists j_0\in \{1,2,\ldots,N,\} \ s.t. \  \tau_{j_0}>0 \}$$
is $\eta=(\eta_1,\ldots,\eta_N)^T.$

Thus $\lim\limits_{x\rightarrow-\infty}f_i(x)=0,$ $\lim\limits_{x\rightarrow+\infty}f_i(x)=\eta_i,$ $f_i\in C(\mathbb{R})$ and the functions $f_i(x),$ $i=\overline{1,N}$ are monotonically non-decreasing on $\mathbb{R}.$ Inequality  \eqref{Khachatryan34}  follows from these facts.
The lemma is proved.
\end{proof}
Repeating similar arguments, we can prove
\begin{lemma}
Let conditions  $1), 2), a), b), I)-III),$ $\alpha_i^+=0$ and $\alpha_i^->0,$ $i=\overline{1,N}$ are fulfilled. Then any monotone non-increasing solution  $f\in\Omega$ of problem  \eqref{Khachatryan1}, \eqref{Khachatryan5} satisfies the inequality:
$$f_i(x)\leq \eta_i,\quad x\in\mathbb{R},\quad i=\overline{1,N}.$$
Moreover, $\lim\limits_{x\rightarrow-\infty}f_i(x)=\eta_i,$ $i=\overline{1,N}.$
\end{lemma}
The following key lemma plays an important role in our further arguments:
\begin{lemma}
Under the conditions of Lemma~2.4, for an arbitrary solution  $f\in\Omega$ of problem  \eqref{Khachatryan1}, \eqref{Khachatryan5} the inclusions $G_i(f_i)\in L_1(0,+\infty),$ $i=\overline{1,N}$ take place
\end{lemma}
\begin{proof}
First, note that estimate  \eqref{Khachatryan26} immediately implies that $G_i(f_i)-f_i\in L_1(r,+\infty),$ $i=\overline{1,N}.$ Since $f_i(x)\geq0,$ $x\in\mathbb{R},$ $f_i\in C(\mathbb{R})$ and $G_i\in C[0,+\infty),$ $i=\overline{1,N},$ then $G_i(f_i)-f_i\in L_1(0,r),$ $i=\overline{1,N}.$ Consequently, $G_i(f_i)-f_i\in L_1(0,+\infty),$ $i=\overline{1,N}.$ Using Lemma~2.4 and the obvious inequality
$$0\leq G_i(f_i(x))\leq |G_i(f_i(x))-f_i(x)|+f_i(x),\quad x\in [0,+\infty),\quad i=\overline{1,N}$$
we come to the end of the proof of the lemma.
\end{proof}
Similarly, we prove
\begin{lemma}
Under the conditions of Lemma~2.5, for an arbitrary solution  $f\in\Omega$ of problem  \eqref{Khachatryan1}, \eqref{Khachatryan5} the inclusions $G_i(f_i)\in L_1(-\infty,0),\, i=\overline{1,N}$ hold.
\end{lemma}
\section{On the solvability of system  \eqref{Khachatryan1}}
\subsection{Successive approximations for system  \eqref{Khachatryan1}.}
Consider the following successive approximations \eqref{Khachatryan1}
\begin{equation}\label{Khachatryan37}
\begin{array}{c}
\displaystyle f_i^{(n+1)}(x)=\sum\limits_{j=1}^N\int\limits_{-\infty}^\infty K_{ij}(x-t) \mu_j (t)G_j(f_j^{(n)}(t))dt,\\
\displaystyle f_i^{(0)}(x)\equiv\xi_i,\quad x\in\mathbb{R},\quad i=\overline{1,N},\quad n=0,1,2,\ldots,
\end{array}
\end{equation}
where the numbers  $\xi_i>\eta_i,$ $i=\overline{1,N}$ are uniquely determined from the system of nonlinear algebraic equations  \eqref{Khachatryan3}. By induction on  $n,$ it is easy to prove that
\begin{equation}\label{Khachatryan38}
f_i^{(n)}(x)\geq \eta_i, \quad x\in\mathbb{R},\quad i=\overline{1,N},\quad n=0,1,2,\ldots,
\end{equation}
\begin{equation}\label{Khachatryan39}
f_i^{(n)}(\cdot)\in C(\mathbb{R}),\quad i=\overline{1,N},\quad n=0,1,2,\ldots.
\end{equation}
We now prove that
\begin{equation}\label{Khachatryan40}
f_i^{(n)}(x)\quad \mbox{decrease monotonically in } \quad n,\quad  x\in\mathbb{R},\quad i=\overline{1,N}.
\end{equation}
First, we check the estimates  $f_i^{(1)}(x)\leq f_i^{(0)}(x),$ $x\in\mathbb{R},$ $i=\overline{1,N}.$ Using conditions  $1), 2), a), b), I), II)$ and the fact that the vector  $\xi=(\xi_1,\ldots,\xi_N)^T$ is a solution of system  \eqref{Khachatryan3}, from \eqref{Khachatryan1} we have
$$f_i^{(1)}(x)\leq \sum\limits_{j=1}^N G_j(\xi_j)\left(\int\limits_{-\infty}^\infty K_{ij}(x-t)dt+ \int\limits_{-\infty}^\infty (\mu_j(t)-1)dt\sup\limits_{\tau\in\mathbb{R}}(K_{ij}(\tau))\right)=$$
$$=\sum\limits_{j=1}^N  (a_{ij}+b_{ij})G_j(\xi_j)=\xi_i=f_i^{(0)}(x),\quad x\in\mathbb{R},\quad i=\overline{1,N}.$$
Assume that  $f_i^{(n)}(x)\leq f_i^{(n-1)}(x),$ $x\in\mathbb{R},$ $i=\overline{1,N}$ for some natural  $n.$ Then, considering  \eqref{Khachatryan38} and conditions  $1), 2), a), b), I), II)$ from \eqref{Khachatryan1} we obtain that  $f_i^{(n+1)}(x)\leq f_i^{(n)}(x),$ $x\in\mathbb{R},$ $i=\overline{1,N}.$ Thus, based on  \eqref{Khachatryan38}-\eqref{Khachatryan40} we can assert that the sequence of continuous vector functions  $f^{(n)}(x)=(f_1^{(n)}(x),\ldots,f_N^{(n)}(x))^T,$ $n=0,1,2,\ldots,$  has a pointwise limit when  $n\rightarrow\infty: \lim\limits_{n\rightarrow\infty}f_i^{(n)}(x)=f_i(x),$ $i=\overline{1,N},$ and by virtue of conditions  $1), 2), a), b), I), II)$ and B. Levy's theorem  (see \cite{kol25}) the limit vector function  $f(x)=(f_1(x),\ldots,f_N(x))^T,$ satisfies system  \eqref{Khachatryan1}. From \eqref{Khachatryan38} and \eqref{Khachatryan40} it also follows that the two-sided inequality
\begin{equation}\label{Khachatryan41}
\eta_i\leq f_i(x)\leq \xi_i, \quad x\in\mathbb{R},\quad i=\overline{1,N},
\end{equation}
moreover, by virtue of  \eqref{Khachatryan15} $f_i\in C(\mathbb{R}),$ $i=\overline{1,N}.$

In the next section, under additional constraint  $IV)$ on the nonlinearities  $G_j(u),$ $j=\overline{1,N},$ we obtain an uniform estimate for the difference $f_i^{(n)}(x)-f(x),$ $x\in\mathbb{R},$ $ i=\overline{1,N},$ $n=1,2,\ldots,$ from which the uniform convergence of successive approximations  \eqref{Khachatryan37} to the continuous solution  $f(x)=(f_1(x),\ldots,f_N(x))^T,$   follow, with the rate of some geometric progression.
\subsection{Uniform convergence of successive approximations  \eqref{Khachatryan37}.}
Assume that condition  $IV)$ is also satisfied and consider the following functions
$$B_i(x):=\frac{f_i^{(1)}(x)}{f_i^{(0)}(x)},\quad x\in\mathbb{R},\quad i=\overline{1,N}.$$
From \eqref{Khachatryan38}-\eqref{Khachatryan40} it immediately follows that
\begin{equation}\label{Khachatryan42}
\frac{\eta_i}{\xi_i}\leq B_i(x)\leq 1, \quad x\in\mathbb{R},\quad i=\overline{1,N},
\end{equation}
\begin{equation}\label{Khachatryan43}
B_i\in C(\mathbb{R}),\quad i=\overline{1,N}.
\end{equation}
Let $\sigma:=\min\limits_{1\leq i\leq N}\left(\frac{\eta_i}{\xi_i}\right).$ Since $\xi_i>\eta_i>0,$ $i=\overline{1,N},$ then $\sigma\in(0,1).$

From \eqref{Khachatryan42}  it follows that
\begin{equation}\label{Khachatryan44}
\sigma f_j^{(0)}(t)\leq f_j^{(1)}(t)\leq f_j^{(0)}(t), \quad t\in\mathbb{R},\quad j=\overline{1,N}.
\end{equation}
Taking into account conditions  $I), II)$ and $IV)$ from \eqref{Khachatryan44} we arrive at inequalities
\begin{equation}\label{Khachatryan45}
\varphi(\sigma)G_j(f_j^{(0)}(t))\leq G_j(\sigma f_j^{(0)}(t))\leq G_j(f_j^{(1)}(t))\leq G_j(f_j^{(0)}(t)), \quad t\in\mathbb{R},\quad j=\overline{1,N}.
\end{equation}
Multiply both parts of \eqref{Khachatryan45} by the functions  $K_{ij}(x,t)\mu_j(t),$ $x,t\in \mathbb{R},$ $i,j=\overline{1,N},$ then integrate the obtained inequalities by  $t$ from $-\infty$ to $+\infty.$ As a result, we get
\begin{equation}\label{Khachatryan46}
\begin{array}{c}
\displaystyle\varphi(\sigma)\int\limits_{-\infty}^\infty K_{ij}(x,t)\mu_j(t) G_j(f_j^{(0)}(t))dt\leq  \int\limits_{-\infty}^\infty K_{ij}(x,t)\mu_j(t) G_j(f_j^{(1)}(t))dt\leq \\
\displaystyle \leq\int\limits_{-\infty}^\infty K_{ij}(x,t)\mu_j(t)G_j(f_j^{(0)}(t))dt, \quad x\in\mathbb{R},\quad i,j=\overline{1,N}.
\end{array}
\end{equation}
Summing both parts of  \eqref{Khachatryan46} over all  $j=\overline{1,N}$ and considering  \eqref{Khachatryan37} we arrive at
\begin{equation}\label{Khachatryan47}
\varphi(\sigma) f_i^{(1)}(x)\leq f_i^{(2)}(x)\leq f_i^{(1)}(x), \quad x\in\mathbb{R},\quad i=\overline{1,N}.
\end{equation}
Since $\varphi(\sigma)\in(0,1),$ $f_i^{(1)}, f_i^{(2)}\in[\eta_i,\xi_i],$ $i=\overline{1,N},$ then again taking into account conditions  $I), II),$ and $IV)$ from \eqref{Khachatryan47} we obtain
\begin{equation}\label{Khachatryan48}
\varphi(\varphi(\sigma))G_j(f_j^{(1)}(t))\leq G_j(\varphi(\sigma) f_j^{(1)}(t))\leq G_j(f_j^{(2)}(t))\leq G_j(f_j^{(1)}(t)), \ t\in\mathbb{R},\ j=\overline{1,N}.
\end{equation}
Again  multiply both parts of  \eqref{Khachatryan48} by $K_{ij}(x,t)\mu_j(t),$ $x,t\in\mathbb{R},\quad i,j=\overline{1,N},$ integrate over  $t$ on $\mathbb{R}$ and sum the obtained inequalities over all $j=\overline{1,N}.$ As a result, we arrive at the following chain of inequalities:
\begin{equation}\label{Khachatryan49}
\varphi(\varphi(\sigma)) f_i^{(2)}(x)\leq f_i^{(3)}(x)\leq f_i^{(2)}(x), \quad x\in\mathbb{R},\quad i=\overline{1,N}.
\end{equation}
Continuing this process in the $n$-th step, we obtain
\begin{equation}\label{Khachatryan50}
\underbrace{\varphi(\varphi\ldots\varphi(\sigma))}_n f_i^{(n)}(x)\leq f_i^{(n+1)}(x)\leq f_i^{(n)}(x), \quad x\in\mathbb{R},\quad i=\overline{1,N}.
\end{equation}
From the properties of the function  $\varphi$ it immediately follows that  (see Fig. 2)
\begin{equation}\label{Khachatryan51}
\varphi(\sigma)\geq k\sigma+1-k,
\end{equation}
where $k:=\dfrac{1-\varphi(\frac{\sigma}{2})}{1-\frac{\sigma}{2}}\in(0,1).$

Since $\sigma\in(0,1),$ $k\in(0,1),$ then $k\sigma+1-k\in(\sigma,1).$

\begin{figure}[!h]
  \centering
  \includegraphics[width=3.5in]{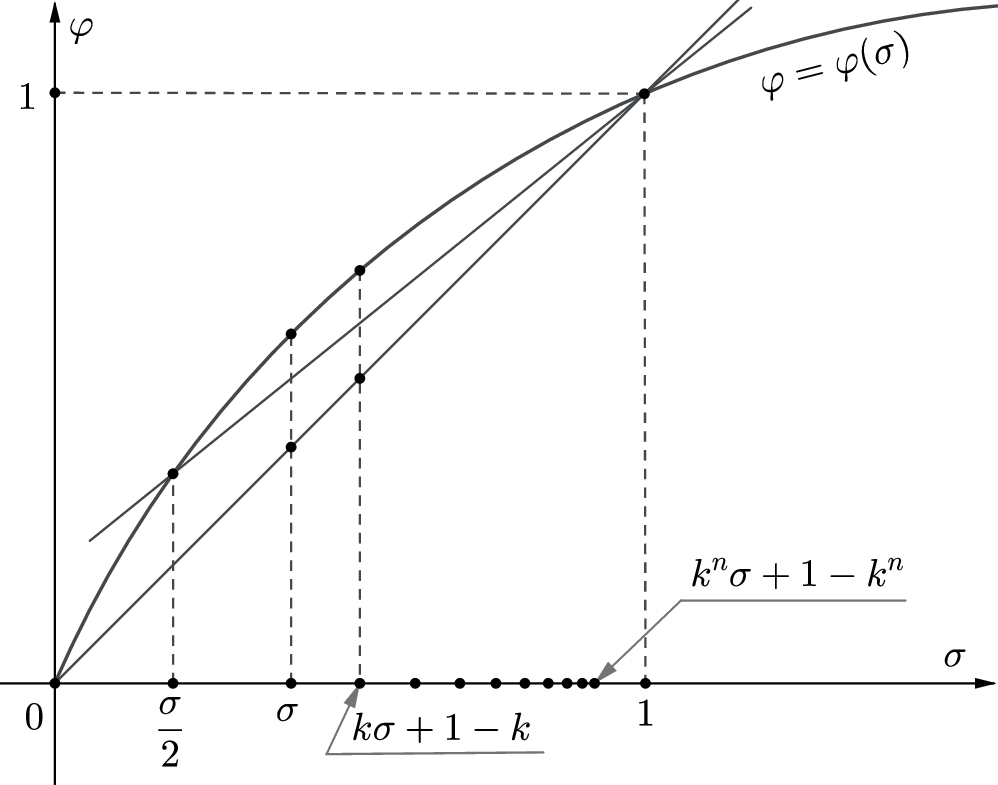}\\
  \caption{Graph of the function $\varphi(\sigma)$.}
\end{figure}

Therefore, from  \eqref{Khachatryan51} due to the monotonicity of  $\varphi$ we obtain
$$\varphi(\varphi(\sigma))\geq \varphi(k\sigma+1-k)\geq k(k\sigma+1-k)+1-k=k^2\sigma+1-k^2.$$
Continuing this process in the $n$-th step, we obtain
\begin{equation}\label{Khachatryan52}
\underbrace{\varphi(\varphi\ldots\varphi(\sigma))}_n\geq k^n\sigma+1-k^n.
\end{equation}
From \eqref{Khachatryan50} and \eqref{Khachatryan52} it follows that
\begin{equation}\label{Khachatryan53}
0\leq f_i^{(n)}(x)- f_i^{(n+1)}(x)\leq k^n(1-\sigma), \quad x\in\mathbb{R},\quad i=\overline{1,N},\quad n=1,2,\ldots.
\end{equation}
We now write inequality  \eqref{Khachatryan53} for the values of indices  $n+1, n+2, \ldots, n+m:$
$$0\leq f_i^{(n+1)}(x)- f_i^{(n+2)}(x)\leq k^{n+1}(1-\sigma), \quad x\in\mathbb{R},\quad i=\overline{1,N}$$
$$0\leq f_i^{(n+2)}(x)- f_i^{(n+3)}(x)\leq k^{n+2}(1-\sigma), \quad x\in\mathbb{R},\quad i=\overline{1,N}$$
$$\ldots\ldots\ldots\ldots\ldots\ldots\ldots\ldots\ldots\ldots\ldots\ldots\ldots\ldots
\ldots\ldots\ldots\ldots\ldots$$
$$0\leq f_i^{(n+m)}(x)- f_i^{(n+m+1)}(x)\leq k^{n+m}(1-\sigma), \quad x\in\mathbb{R},\quad i=\overline{1,N}.$$
Summing these  inequalities and inequality  \eqref{Khachatryan53} we come to  the estimate
\begin{equation}\label{Khachatryan54}
\begin{array}{c}
\displaystyle 0\leq f_i^{(n)}(x)- f_i^{(n+m+1)}(x)\leq (1-\sigma)k^{n}(1+\ldots+k^m)\leq \frac{k^n(1-\sigma)}{1-k},\\
\displaystyle x\in\mathbb{R},\quad i=\overline{1,N},\quad  n=1,2,\ldots,\quad m=0,1,2,\ldots.
\end{array}
\end{equation}
In \eqref{Khachatryan54} letting $m\rightarrow\infty$ we obtain
\begin{equation}\label{Khachatryan55}
0\leq f_i^{(n)}(x)- f_i(x)\leq \frac{k^n(1-\sigma)}{1-k},\quad x\in\mathbb{R},\quad i=\overline{1,N},\quad  n=1,2,\ldots.
\end{equation}
From \eqref{Khachatryan55} taking into account that $k\in(0,1)$ it follows that successive approximations converge uniformly to a continuous solution $f(x)=(f_1(x),\ldots,f_N(x))^T$ on $\mathbb{R}$ of the system  \eqref{Khachatryan1}. By repeating similar reasoning as in the proof of Lemma~2.6 and considering  \eqref{Khachatryan41} we can confirm  that under conditions  $1), 2), a), b)$ and $I)-III)$ the solution  $f(x)=(f_1(x),\ldots,f_N(x))^T$ of the system \eqref{Khachatryan1} has the integral asymptotics  $f_i-\eta_i\in L_1^0(\mathbb{R}),$ $i=\overline{1,N}.$

Thus, based on the above, we arrive at the following theorem:
\begin{theorem}
Under conditions  $1), 2), a), b)$ and $I)-III)$ the system of nonlinear integral equations  \eqref{Khachatryan1} has a positive bounded and continuous on $\mathbb{R}$ solution $f(x)=(f_1(x),\ldots,f_N(x))^T,$ besides $f_i-\eta_i\in L_1^0(\mathbb{R}),$ $i=\overline{1,N}$ and double inequality  \eqref{Khachatryan41} holds.  Moreover, if condition  $IV),$ is also satisfied, then  $f(x)$ is a uniform limit of successive approximations  \eqref{Khachatryan37} and estimate  \eqref{Khachatryan55} holds.
\end{theorem}
\subsection{On the uniqueness of a solution to system  \eqref{Khachatryan1}.}
The following theorem holds:
\begin{theorem}
Let conditions  $1), 2), a), b)$ and $I)-IV)$ are fulfilled. Then system  \eqref{Khachatryan1} in
the following class of vector functions
\begin{equation}\label{Khachatryan56}
\mathfrak{M}:=\{f\in\Omega: \,\, \mbox{there exists}\,\,\, j_0\in\{1,2,\ldots,N\} \,\,\mbox{such that }\,\, \alpha_{j_0}^\pm>0\}
\end{equation}
cannot have more than one solution.
\end{theorem}
\begin{proof}
Assume the opposite: system \eqref{Khachatryan1} in addition to the solution  $f(x)=(f_1(x),\ldots,f_N(x))^T$ constructed using successive approximations \eqref{Khachatryan37}, also has another solution  $f^*(x)=(f_1^*(x),\ldots,f_N^*(x))^T\in \mathfrak{M}.$ First we prove that
\begin{equation}\label{Khachatryan57}
f_i^*(x)\leq f_i(x),\quad x\in\mathbb{R},\quad i=\overline{1,N}.
\end{equation}
To this end, we verify that the following inequalities hold:
\begin{equation}\label{Khachatryan58}
f_i^*(x)\leq f_i^{(n)}(x),\quad x\in\mathbb{R},\quad i=\overline{1,N},\quad n=0,1,2,\ldots.
\end{equation}
In the case  $n=0$ inequalities \eqref{Khachatryan58} immediately follow from Lemma~2.1 and the definition of the zero approximation in iterations  \eqref{Khachatryan37}. Assume $f_i^*(x)\leq f_i^{(n)}(x),$ $x\in\mathbb{R},$ $i=\overline{1,N},$ for some natural  $n.$ Then, taking into account conditions  $1), a), I), II)$ from \eqref{Khachatryan37} and \eqref{Khachatryan1} we have
$$f_i^{(n+1)}(x)\geq \sum\limits_{j=1}^N\int\limits_{-\infty}^\infty K_{ij}(x,t) \mu_j (t)G_j(f_j^*(t))dt=f_i^*(x),\quad x\in\mathbb{R},\quad i=\overline{1,N}.$$
In \eqref{Khachatryan58} tending $n\rightarrow\infty$ we arrive at  \eqref{Khachatryan57}. On the other hand, according to Lemma~2.3
\begin{equation}\label{Khachatryan59}
f_i^*(x)\geq\eta_i, \quad x\in\mathbb{R},\quad i=\overline{1,N}.
\end{equation}
From \eqref{Khachatryan41}, \eqref{Khachatryan57} and \eqref{Khachatryan59} it follows that
$$\frac{\eta_i}{\xi_i}\leq \frac{f_i^*(x)}{f_i(x)}\leq 1, \quad x\in\mathbb{R},\quad i=\overline{1,N},$$
in particular, we obtain the estimate
\begin{equation}\label{Khachatryan60}
\sigma f_i(x)\leq f_i^*(x)\leq f_i(x), \quad x\in\mathbb{R},\quad i=\overline{1,N}.
\end{equation}
Then, following the same reasoning as in the proof of Theorem~3.1 (see section 3.2), we arrive at the inequality:
\begin{equation}\label{Khachatryan61}
0\leq f_i(x)-f_i^*(x)\leq k^n(1-\sigma), \quad x\in\mathbb{R},\quad i=\overline{1,N},\quad n=1,2,\ldots.
\end{equation}
In \eqref{Khachatryan61} tending $n\rightarrow\infty$ and considering the fact that  $k\in(0,1)$ we get $f_i(x)=f_i^*(x),$ $x\in\mathbb{R},$ $i=\overline{1,N}.$   The theorem is proved.
\end{proof}
\section{Absence of nontrivial solutions to system  \eqref{Khachatryan1}. Examples.}
\subsection{Theorem on the absence of a nontrivial solution to system  \eqref{Khachatryan1} in the case $\alpha_i^\pm=0,$ $i=\overline{1,N}.$}
\begin{theorem}
Let conditions  $1), 2), a), b)$ and $I)-IV)$ be satisfied. If  $\alpha_i^\pm=0,$ $i=\overline{1,N},$ then system  \eqref{Khachatryan1} in the class  $\Omega$ has only the trivial solution $f_i(x)\equiv0,$ $x\in\mathbb{R},$ $i=\overline{1,N}.$
\end{theorem}
\begin{proof}
Assume the opposite: there exists a nontrivial solution $\tilde{f}(x)=(\tilde{f}_1(x),\ldots,\tilde{f}_N(x))^T\in\Omega$  with the properties  $\lim\limits_{x\rightarrow\pm\infty}\tilde{f}_i(x)=0,$ $i=\overline{1,N}.$ Then, according to Lemmas~2.9 and 2.10
\begin{equation}\label{Khachatryan62}
G_i(\tilde{f}_i)\in L_1(\mathbb{R}),\quad i=\overline{1,N}.
\end{equation}
Let $f(x)=(f_1(x),\ldots,f_N(x))^T$ be the solution of system  \eqref{Khachatryan1} constructed using successive approximations  \eqref{Khachatryan37}. Then, using similar reasoning as in the proof of Theorem~3.2, we can verify the validity of the inequality
\begin{equation}\label{Khachatryan63}
\tilde{f}_i(x)\leq f_i(x),\quad  x\in\mathbb{R},\quad i=\overline{1,N}.
\end{equation}
Since $\tilde{f}\in\Omega,$ $\tilde{f}_i\in C(\mathbb{R}),$ $i=\overline{1,N}$ (see formula  \eqref{Khachatryan15}) and $\tilde{f}$ is a non-trivial solution of system  \eqref{Khachatryan1}, there exist an index $j_0\in\{1,2,\ldots,N\}$ and numbers $x_0\in\mathbb{R},$ $\delta>0$ such that
\begin{equation}\label{Khachatryan64}
\beta:=\inf\limits_{x\in(x_0-\delta,x_0+\delta)}\tilde{f}_{j_0}(x)>0.
\end{equation}
Considering  \eqref{Khachatryan64}, $I), II), 1)$ from \eqref{Khachatryan1} we have
$$\tilde{f}_i(x)\geq \int\limits_{-\infty}^\infty K_{ij_0}(x-t)\mu_{j_0}(t)G_{j_0}(\tilde{f}_{j_0}(t))dt\geq G_{j_0}(\beta)\int\limits_{x_0-\delta}^{x_0+\delta}K_{ij_0}(x-t)dt>0, \ \  x\in\mathbb{R},\ \ i=\overline{1,N}.$$
Thus, we obtain that
\begin{equation}\label{Khachatryan65}
\tilde{f}_i(x)>0, \quad  x\in\mathbb{R},\quad i=\overline{1,N}.
\end{equation}
Since for the solution  $\tilde{f}$ the limit values  $\alpha_i^\pm=0,$ $i=\overline{1,N},$ then according to Theorem~3.1 $\tilde{f}(x)\not\equiv f(x),$ $x\in\mathbb{R}.$ Therefore, taking into account the continuity on  $\mathbb{R}$ of the vector functions  $f$ and $\tilde{f}$ we can assert that there exist  $j^*\in\{1,2,\ldots,N\},$ $x^*\in\mathbb{R}$ and $\delta^*>0$ such that
\begin{equation}\label{Khachatryan66}
\tilde{f}_{j^*}(x)\neq f_{j^*}(x),\quad x\in(x^*-\delta^*,x^*+\delta^*).
\end{equation}
Considering   \eqref{Khachatryan63} from \eqref{Khachatryan66} we arrive at the strict inequality
\begin{equation}\label{Khachatryan67}
\tilde{f}_{j^*}(x)< f_{j^*}(x),\quad x\in(x^*-\delta^*,x^*+\delta^*).
\end{equation}
We denote by  $\Gamma_j$ the following measurable sets:
$$\Gamma_j:=\{x\in\mathbb{R}: \tilde{f}_j(x)<f_j(x)\},\quad j=\overline{1,N}.$$
Obviously, $\Gamma_{j^*}\neq\emptyset$ and $mes \Gamma_{j^*}>0,$ since $(x^*-\delta^*, x^*+\delta^*)\subset\Gamma_{j^*}.$

From the obvious inequalities
$0\leq \mu_i(x)G_i(\tilde{f}_i(x))\leq (\mu_i(x)-1)G_i(\xi_i)+G_i(\tilde{f}_i(x)),$ $x\in\mathbb{R},$ $i=\overline{1,N},$ due to  \eqref{Khachatryan62} and conditions  $a), b)$ it follows that
\begin{equation}\label{Khachatryan68}
\mu_i(x)G_i(\tilde{f}_i(x))\in L_1(\mathbb{R}),\quad i=\overline{1,N}.
\end{equation}
Consider the following difference, taking into account \eqref{Khachatryan1} and \eqref{Khachatryan63}:
\begin{equation}\label{Khachatryan69}
0\leq f_i(x)-\tilde{f}_i(x)=\sum\limits_{j=1}^N\int\limits_{-\infty}^\infty K_{ij}(x-t)\mu_j(t)(G_j(f_j(t))-G_j(\tilde{f}_j(t)))dt,\ \  x\in\mathbb{R},\ \ i=\overline{1,N}.
\end{equation}
We multiply both parts of  \eqref{Khachatryan69} by functions  $\mu_i(x)G_i(\tilde{f}_i(x)),$ $i=\overline{1,N}.$ Then, considering  \eqref{Khachatryan68} and conditions  $1), 2), a), b), I)$ and $II),$ as well as Lemma~2.1, we integrate the obtained
inequality over $x$ on $\mathbb{R}$ and sum over all $i=\overline{1,N}.$ As a result, using Fubini's theorem and condition $1),$ we have
$$0\leq \sum\limits_{i=1}^N\int\limits_{-\infty}^\infty \mu_{i}(x)G_i(\tilde{f}_i(x))(f_i(x)-\tilde{f}_i(x))dx=$$
$$= \sum\limits_{i=1}^N\int\limits_{-\infty}^\infty \mu_{i}(x)G_i(\tilde{f}_i(x))\sum\limits_{j=1}^N\int\limits_{-\infty}^\infty K_{ij}(x-t)\mu_j(t)(G_j(f_j(t))-G_j(\tilde{f}_j(t)))dtdx=$$
$$= \sum\limits_{j=1}^N\int\limits_{-\infty}^\infty \mu_{j}(t)(G_j(f_j(t))-G_j(\tilde{f}_j(t)))\sum\limits_{i=1}^N \int\limits_{-\infty}^\infty K_{ij}(t-x)
\mu_i(x)G_i(\tilde{f}_i(x))dxdt=$$
$$=\sum\limits_{j=1}^N\int\limits_{-\infty}^\infty \mu_{j}(t)(G_j(f_j(t))-G_j(\tilde{f}_j(t)))\sum\limits_{i=1}^N \int\limits_{-\infty}^\infty K_{ji}(t-x)\mu_i(x)G_i(\tilde{f}_i(x))dxdt=$$
$$=\sum\limits_{j=1}^N\int\limits_{-\infty}^\infty \mu_{j}(t)(G_j(f_j(t))-G_j(\tilde{f}_j(t)))\tilde{f}_j(x)dx$$
or by virtue of \eqref{Khachatryan65}
\begin{equation}\label{Khachatryan70}
\sum\limits_{i=1}^N\int\limits_{-\infty}^\infty \mu_i(x)\tilde{f}_i(x)\left(\frac{G_i(\tilde{f}_i(x))}{\tilde{f}_i(x)}
(f_i(x)-\tilde{f}_i(x))-(G_i(f_i(x))-G_i(\tilde{f}_i(x))) \right)dx=0.
\end{equation}
Since $\mathbb{R}\backslash\Gamma_i= \{x\in\mathbb{R}: f_i(x)=\tilde{f}_i(x)\},$ $i=\overline{1,N},$ then \eqref{Khachatryan70} can be rewritten in the following form
\begin{equation}\label{Khachatryan71}
\sum\limits_{i=1}^N\int\limits_{\Gamma_i} \mu_i(x)\tilde{f}_i(x)(f_i(x)-\tilde{f}_i(x)) \left(\frac{G_i(\tilde{f}_i(x))}{\tilde{f}_i(x)}
-\frac{G_i(f_i(x))-G_i(\tilde{f}_i(x))}{f_i(x)-\tilde{f}_i(x)}\right)dx=0.
\end{equation}
From the definition of sets  $\Gamma_j,$ $j=\overline{1,N},$ by virtue of conditions  $I)-III)$ we can state that  (see Fig. 3)
\begin{equation}\label{Khachatryan72}
\frac{G_j(\tilde{f}_j(x))}{\tilde{f}_j(x)}
>\frac{G_j(f_j(x))-G_j(\tilde{f}_j(x))}{f_j(x)-\tilde{f}_j(x)}, \quad x\in \Gamma_j,\quad j=\overline{1,N}.
\end{equation}

\begin{figure}[!h]
  \centering
  \includegraphics[width=4in]{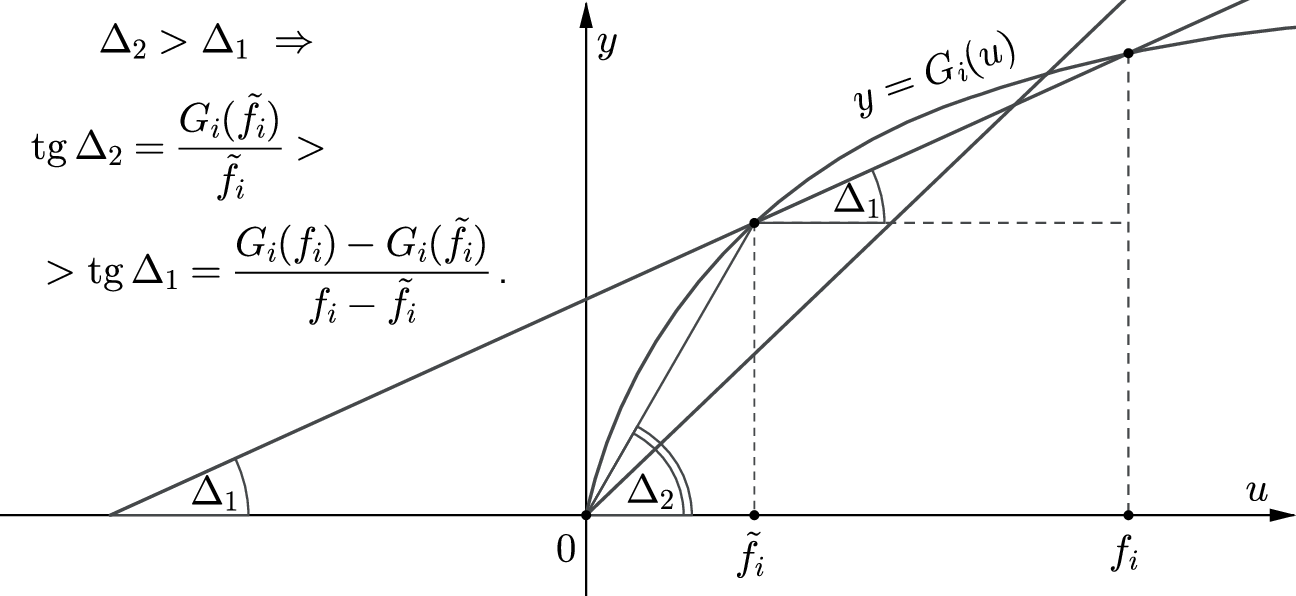}\\
  \caption{Intersection of the graph of function $y=G_i(u)$ with a line passing through points $(f_i,G_i(f_i))$ and $(\tilde{f}_i,G_i(\tilde{f}_i))$.}
\end{figure}

Considering \eqref{Khachatryan65}, \eqref{Khachatryan72}, condition  $a)$ and the definition of sets $\Gamma_j,$ $j=\overline{1,N}$ from \eqref{Khachatryan71} we obtain
$$\int\limits_{\Gamma_{j^*}} \mu_{j^*}(x)\tilde{f}_{j^*}(x)(f_{j^*}(x)-\tilde{f}_{j^*}(x)) \left(\frac{G_{j^*}(\tilde{f}_{j^*}(x))}{\tilde{f}_{j^*}(x)}
-\frac{G_{j^*}(f_{j^*}(x))-G_{j^*}(\tilde{f}_{j^*}(x))}{f_{j^*}(x)-\tilde{f}_{j^*}(x)}\right)dx\leq0.$$

However, the last inequality is impossible since  \eqref{Khachatryan72} holds for  $i=j^*,$ $\mu_{j^*}(x)>1,$ $f_{j^*}(x)>\tilde{f}_{j^*}(x)>0,$ $x\in \Gamma_{j^*}$ and $mes \Gamma_{j^*}>0.$ From the resulting contradiction, we come to the completion of the proof of the formulated theorem.
\end{proof}
\subsection{Possible absence of non-negative solutions of system  \eqref{Khachatryan1} in cases $\alpha_i^-=0,$ $\alpha_i^+>0,$ $i=\overline{1,N}$ and $\alpha_i^->0,$ $\alpha_i^+=0,$ $i=\overline{1,N}.$}
We note immediately that in cases  $\alpha_i^-=0,$ $\alpha_i^+>0,$ $i=\overline{1,N}$ and $\alpha_i^->0,$ $\alpha_i^+=0,$ $i=\overline{1,N}$ the question of the existence or absence of a solution for problem \eqref{Khachatryan1}, \eqref{Khachatryan5} in the class of functions  $\Omega,$ when conditions  $1), 2), a), b), I)-IV),$ are satisfied, remains an open problem. However, we were able to prove that in the indicated cases, problem \eqref{Khachatryan1}, \eqref{Khachatryan5} cannot have monotone solutions in the class  $\Omega.$

More precisely, the following theorems are valid.
\begin{theorem}
Let conditions  $1), 2), a), b), I)-III)$ are fulfilled and $\alpha_i^-=0,$ $\alpha_i^+>0,$ $i=\overline{1,N}.$ Then problem  \eqref{Khachatryan1}, \eqref{Khachatryan5} cannot have a monotonically non-decreasing solution from the class  $\mathbb{R}$ on the set  $\Omega.$
\end{theorem}
\begin{theorem}
Let conditions  $1), 2), a), b), I)-III)$ be satisfied and $\alpha_i^->0,$ $\alpha_i^+=0,$ $i=\overline{1,N}.$ Then problem \eqref{Khachatryan1}, \eqref{Khachatryan5} cannot have a monotonically non-increasing solution from the class  $\mathbb{R}$ on the set  $\Omega.$
\end{theorem}
Let us prove Theorem~4.2. The proof of Theorem~4.3 is carried out in a similar way.
\begin{proof}
Assume the opposite: under the conditions  $\alpha_i^-=0,$ $\alpha_i^+>0,$ $i=\overline{1,N}$ problem \eqref{Khachatryan1}, \eqref{Khachatryan5} has a monotonically non-decreasing solution $f\in\Omega.$ Then it is clear that  $f_i\in C(\mathbb{R}),$ $i=\overline{1,N}$ (see \eqref{Khachatryan15}). From Lemma~2.7 and the definition of the class  $\Omega$ it immediately follows that
\begin{equation}\label{Khachatryan73}
0\leq f_i(x)\leq \eta_i, \quad x\in\mathbb{R},\quad i=\overline{1,N}.
\end{equation}
\begin{equation}\label{Khachatryan74}
\lim\limits_{x\rightarrow+\infty}f_i(x)=\eta_i,\quad i=\overline{1,N}.
\end{equation}
Therefore, there exists a number  $r>0$ such that for  $x>r$ we have
\begin{equation}\label{Khachatryan75}
f_i(x)>\frac{\eta_i}{2},\quad i=\overline{1,N}.
\end{equation}
Then from  \eqref{Khachatryan1} taking into account  \eqref{Khachatryan75}, $1), a), I)$ and $II)$ we have
$$f_i(x)\geq \sum\limits_{j=1}^N\int\limits_{r}^\infty K_{ij}(x-t)\mu_j(t)G_j(f_j(t))dt\geq$$
$$\geq \sum\limits_{j=1}^N\ G_j(\frac{\eta_j}{2})\int\limits_{r}^\infty K_{ij}(x-t)\mu_j(t)dt>0,\ \  x\in\mathbb{R},\ \ i=\overline{1,N}.$$
Therefore, we obtain that
\begin{equation}\label{Khachatryan76}
f_i(x)>0, \quad x\in\mathbb{R},\quad i=\overline{1,N}.
\end{equation}
On the other hand, since  $\alpha_i^-=0,$ $i=\overline{1,N}$ and $f\in\Omega,$ there exists  $r_0>0$ such that $x<-r_0$
\begin{equation}\label{Khachatryan77}
f_i(x)<\frac{\eta_i}{2},\quad i=\overline{1,N}.
\end{equation}
Thus, we obtain an important two-sided inequality for further reasoning:
\begin{equation}\label{Khachatryan78}
0<f_i(x)<\frac{\eta_i}{2},\quad x\in(-\infty,-r_0),\quad i=\overline{1,N}.
\end{equation}
Since $\alpha_i^-=0,$ then according to Lemma~2.5
\begin{equation}\label{Khachatryan79}
f_i\in L_1(-\infty,0),\quad i=\overline{1,N}.
\end{equation}
Next, using the same reasoning as in the proof of Lemma~2.6, we can verify that
\begin{equation}\label{Khachatryan80}
\eta_i-f_i\in L_1(0,+\infty),\quad i=\overline{1,N}.
\end{equation}
From \eqref{Khachatryan76}, \eqref{Khachatryan73}, $a)$ and $I)-III)$ it  follows that
$$0\leq (\eta_i-G_i(f_i(x)))\mu_i(x)\leq \eta_i(\mu_i(x)-1)+\eta_i-f_i(x),\,\, x\in\mathbb{R},\quad i=\overline{1,N}.$$
From the last inequality, condition  $b)$ and inclusion  \eqref{Khachatryan80} it follows that
\begin{equation}\label{Khachatryan81}
(\eta_i-G_i(f_i(x)))\mu_i(x)\in L_1(0,+\infty),\quad i=\overline{1,N}.
\end{equation}

Multiply both parts of \eqref{Khachatryan1} by the function $(\eta_i-G_i(f_i(x)))\mu_i(x)$,  considering \eqref{Khachatryan79}, \eqref{Khachatryan81}, \eqref{Khachatryan2}, \eqref{Khachatryan6}, \eqref{Khachatryan73}, $1), 2), a), b), I)$ and $II),$  and integrate with respect to $x$ from $-\infty$ to $+\infty.$ Taking into account conditions $1), 2), a), b),$ according to Fubini’s theorem, we will have
$$\sum\limits_{i=1}^N \int\limits_{-\infty}^\infty (\eta_i-G_i(f_i(x)))\mu_i(x) f_i(x)dx=$$
$$=\sum\limits_{i=1}^N \int\limits_{-\infty}^\infty (\eta_i-G_i(f_i(x)))\mu_i(x)\sum\limits_{j=1}^N \int\limits_{-\infty}^\infty K_{ij}(x-t)\mu_j(t)G_j(f_j(t))dtdx= $$
$$=\sum\limits_{j=1}^N \int\limits_{-\infty}^\infty\mu_j(t)G_j(f_j(t))\sum\limits_{i=1}^N \int\limits_{-\infty}^\infty K_{ji}(t-x)(\eta_i\mu_i(x)-\mu_i(x)G_i(f_i(x)))dxdt> $$
$$>\sum\limits_{j=1}^N \int\limits_{-\infty}^\infty\mu_j(t)G_j(f_j(t))\left(\sum\limits_{i=1}^N a_{ji}\eta_i-f_j(t)\right)dt= \sum\limits_{j=1}^N \int\limits_{-\infty}^\infty\mu_j(t)G_j(f_j(t))(\eta_j-f_j(t))dt$$
or which is the same in virtue of  \eqref{Khachatryan76}
\begin{equation}\label{Khachatryan82}
\sum\limits_{i=1}^N \int\limits_{-\infty}^\infty  \mu_i(x)f_i(x)\left(\eta_i-G_i(f_i(x))-\frac{G_i(f_i(x))}{f_i(x)}(\eta_i-f_i(x))\right)dx>0.
\end{equation}
Consider the sets:
$$\mathcal{D}_i:=\{x\in\mathbb{R}: f_i(x)<\eta_i\},\quad i=\overline{1,N}.$$
By \eqref{Khachatryan78}  $\mathcal{D}_i=\emptyset$ and $mes \mathcal{D}_i=+\infty,$ $i=\overline{1,N}.$

It is clear that inequality  \eqref{Khachatryan82} can be rewritten in the form:
$$\sum\limits_{i=1}^N \int\limits_{\mathcal{D}_i}  \mu_i(x)f_i(x)(\eta_i-f_i(x) \left(\frac{\eta_i-G_i(f_i(x))}{\eta_i-f_i(x)}- \frac{G_i(f_i(x))}{f_i(x)}\right)dx>0,$$
since $\mathcal{D}_i^c:=\{x\in\mathbb{R}: f_i(x)=\eta_i\},\quad i=\overline{1,N}.$
 The last inequality is impossible since
$$\mu_i(x)>1,\quad x\in\mathbb{R},\quad f_i(x)>0,\quad x\in\mathbb{R},\quad i=\overline{1,N},$$
$$f_i(x)<\eta_i,\quad x\in\mathcal{D}_i,\quad i=\overline{1,N}$$
and
$$\frac{G_i(f_i(x))}{f_i(x)}>\frac{\eta_i-G_i(f_i(x))}{\eta_i-f_i(x)},\quad x\in\mathcal{D}_i,\quad i=\overline{1,N}.$$
From the obtained contradiction, we come to the end of the proof of the theorem.
\end{proof}
\subsection{Examples.}
At the end of the paper, we give several illustrative examples of the matrix kernel $K(\tau)=(K_{ij}(\tau))_{i,j=1}^{N\times N},$ functions $\mu_j(t),$ $j=\overline{1,N}$ and nonlinearities  $G_j(u),$ $j=\overline{1,N}.$

\textbf{Examples of the matrix kernel  $K.$}
\begin{enumerate}
  \item [$a_1)$] $K_{ij}(\tau)=\frac{a_{ij}}{\sqrt{\pi}}e^{-\tau^2},$ $\tau\in\mathbb{R},$ where $a_{ij}=a_{ji}>0,$ $i,j=\overline{1,N}$ and the spectral radius of the matrix  $A=(a_{ij})_{i,j=1}^{N\times N}$ is equal to one,
  \item [$a_2)$] $K_{ij}(\tau)=\int\limits_a^b e^{-|\tau|s}L_{ij}(s)ds,$ $\tau\in \mathbb{R},$ where $L_{ij}(s)=L_{ji}(s)>0$ are measurable functions on  $[a,b),$ $0<a<b\leq +\infty,$ and $\int\limits_a^b \frac{L_{ij}(s)}{s}ds<+\infty,$ $i,j=\overline{1,N}$ and the spectral radius of the matrix   $A=\left(\int\limits_a^b \frac{L_{ij}(s)}{s}ds\right)_{i,j=1}^{N\times N}$ is unit.
\end{enumerate}

\textbf{Examples of functions  $\mu_j(t),\ j=\overline{1,N}.$}
\begin{enumerate}
  \item [$b_1)$] $\mu_j(t)=1+\varepsilon_j\dfrac{e^{-|t|}}{\sqrt{|t|}},$ $t\in\mathbb{R},$ $\varepsilon_j>0$ --- are parameters  $j=\overline{1,N},$
  \item [$b_2)$] $\mu_j(t)=1+\dfrac{\varepsilon_j}{(1+t^2)|t|^\alpha},$ $t\in\mathbb{R},$ $\varepsilon_j>0$ ---  are parameters  $j=\overline{1,N},$ and $\alpha\in(0,1).$
\end{enumerate}
Let us now turn to examples of nonlinearities  $G_i(u),$ $i=\overline{1,N}:$

\textbf{Examples of nonlinearities  $G_j(u),$ $j=\overline{1,N}.$}
\begin{enumerate}
  \item [$c_1)$] $G_j(u)=u^\alpha\eta_j^{1-\alpha},$ $u\in[0,+\infty),$ $\alpha\in(0,1)$ --- is a parameter, $j=\overline{1,N},$
  \item [$c_2)$] $G_j(u)=\frac{1}{2}(\sqrt{u\eta_j}+u^\alpha\eta_j^{1-\alpha}),$ $u\in[0,+\infty),$ $\alpha\in(0,1),$ $j=\overline{1,N},$
  \item [$c_3)$] $G_j(u)=\frac{1}{2}(u^\beta\eta_j^{1-\beta}+u^\alpha\eta_j^{1-\alpha}),$ $u\in[0,+\infty),$ $\alpha,\beta\in(0,1)$ --- are the parameters, $j=\overline{1,N},$
  \item [$c_4)$] $G_j(u)=\gamma_j(1-e^{-u^\alpha\eta_j^{1-\alpha}}),$ $u\in[0,+\infty),$ $\alpha\in(0,1),$ $\gamma_j:=\frac{\eta_j}{1-e^{-\eta_j}},$ $j=\overline{1,N}.$
\end{enumerate}
Consider the examples  $c_2)$ and $c_4).$

First, let us check the fulfillment of conditions  $I)-IV)$ for example  $c_2).$ First, it is obvious that  $G_j(0)=0,$ $G_j(\eta_j)=\eta_j,$ $G_j\in C[0,+\infty),$ $j=\overline{1,N}.$ Since $$G_j^\prime(u)=\frac{1}{2}\left( \frac{\sqrt{\eta_j}}{2\sqrt{u}}+\alpha\eta_j^{1-\alpha}u^{\alpha-1}\right)>0,$$ for $u>0,$ $j=\overline{1,N},$
 and $G_j^{\prime\prime}(u)=\frac{1}{2}\alpha(\alpha-1)\eta_j^{1-\alpha}u^{\alpha-2}-
 \frac{\sqrt{\eta_j}}{8}u^{-\frac{3}{2}}<0,$ for $u>0,$ $j=\overline{1,N},$ then  functions  $G_j(u),$ $j=\overline{1,N}$ are monotonically increasing and concave on   $[0,+\infty).$ Finally we check condition $IV).$ As $\varphi(\sigma)$ taking the function $$\varphi(\sigma)=\sigma^{\max (\frac{1}{2},\alpha)},\quad \sigma\in[0,1]$$ we obtain
 $$G_j(\sigma u)=\frac{1}{2}(\sqrt{\sigma u\eta_j}+\sigma^\alpha u^\alpha\eta_j^{1-\alpha})\geq \frac{1}{2} \sigma^{\max (\frac{1}{2},\alpha)}(\sqrt{ u\eta_j}+ u^\alpha\eta_j^{1-\alpha})=$$
 $$= \varphi(\sigma)G_j(u),\quad \sigma\in[0,1],\quad u\in[0,+\infty),\quad j=\overline{1,N}.$$
 Now let us move on to example  $c_4).$ First, note that  $G_j(0)=0,$ $G_j(\eta_j)=\eta_j,$ $G_j\in C[0,+\infty),$ $j=\overline{1,N}.$ Since $G_j^\prime(u)=\gamma_j\alpha u^{\alpha-1}\eta_j^{\alpha-1}e^{-u^\alpha \eta_j^{1-\alpha}}$ and $G_j^{\prime\prime}(u)=\gamma_j\alpha(\alpha-1)u^{\alpha-2}\eta_j^{1-\alpha}e^{-u^\alpha \eta_j^{1-\alpha}}-\gamma_j (\alpha  u^{\alpha-1}\eta_j^{1-\alpha})^2e^{-u^\alpha \eta_j^{1-\alpha}}<0,$ when $u>0,$ $j=\overline{1,N},$ then $G_j(u),$ $j=\overline{1,N}$ are monotonically increasing and concave functions on  $[0,+\infty).$ To verify condition  $IV)$ we first note that for all $t\in[0,1]$ and $u\in[0,+\infty)$ the inequalities
\begin{equation}\label{Khachatryan83}
G_j(tu)\geq tG_j(u),\quad j=\overline{1,N}.
\end{equation}
hold.

Indeed, for  $t=0,$ $u=0$ and $t=1$ inequality \eqref{Khachatryan83} is obviously satisfied. If  $t\in(0,1)$ and $u\in(0,+\infty),$  then it immediately follows from properties  $I)-III)$ that
$$\frac{G_j(tu)}{tu}>\frac{G_j(u)}{u},\quad j=\overline{1,N},$$
whence \eqref{Khachatryan83} follows. Since for functions  $\tilde{G}_j(u):=\gamma_j(1-e^{-u}),$ $j=\overline{1,N}$ conditions $I)-III)$ are obviously satisfied, then
$$\tilde{G}_j(tu)\geq t\tilde{G}_j(u),\quad t\in[0,1],\quad u\in[0,+\infty),\quad j=\overline{1,N}.$$
Therefore, choosing as functions $\varphi(\sigma)=\sigma^\alpha$ we have
$$G_j(\sigma u)=\gamma_j(1-e^{-\sigma^\alpha u^\alpha\eta_j^{1-\alpha}})=\tilde{G}_j(\sigma^\alpha u^\alpha\eta_j^{1-\alpha})\geq \sigma^\alpha \tilde{G}_j( u^\alpha\eta_j^{1-\alpha})=\varphi(\sigma)G_j(u),$$$$\sigma\in[0,1],\quad u\in[0,+\infty),\quad j=\overline{1,N}.$$

\section*{Disclosure statement}

The authors of this work declare that they have no conflicts of interest.

\section*{Funding}

The research of Kh.~A. Khachatryan was supported by the Higher Education and Science Committee of the Republic of Armenia (project no. 23RL-1A027).

\end{document}